\documentclass[11pt,a4paper]{amsart}
\usepackage{amsfonts,amsmath,amssymb,amsthm}
\usepackage{times}
\usepackage{graphics,graphicx}
\usepackage{color}
\usepackage{enumerate}

\usepackage[utf8]{inputenc}
\usepackage{bbm, dsfont}

\usepackage{mathtools}
\mathtoolsset{showonlyrefs}

\numberwithin{equation}{section}

\renewcommand{\epsilon}{\varepsilon}

\DeclareSymbolFont{SY}{U}{psy}{m}{n}
\DeclareMathSymbol{\emptyset}{\mathord}{SY}{'306}

\newcommand{\dmax}{d_{\mathrm{max}}}
\newcommand{\E}{\mathbb{E}}

\newtheorem{theorem}{Theorem}[section]{\bf}{\it}
{\bf}{\it}
\newtheorem{proposition}[theorem]{Proposition}{\bf}{\it}
{\bf}{\it}
{\it}{\rm}
\newtheorem{lemma}[theorem]{Lemma}{\bf}{\it}
{\bf}{\it}
{\bf}{\it}
{\bf}{\it}

\theoremstyle{definition}
\newtheorem{remark}[theorem]{Remark}

\newcommand{\black}{\raisebox{.3ex}{\tiny$\bullet$}}

\title[Almost sure convergence of vertex degree densities in the vertex--splitting model]{Almost sure convergence of vertex degree densities in the vertex--splitting model}

\author[S.\"{O}.~Stef\'{a}nsson]{Sigurdur \"{O}. Stef\'{a}nsson}
 \address{S.\"{O}.~Stef\'{a}nsson, Division of Mathematics, The Science Institute, University of Iceland,
Dunhaga 3 IS-107 Reykjavik, Iceland}
  \email{sigurdur@hi.is}
\author[E.~Th\"ornblad]{Erik Th\"ornblad}		
 \address{E.~Th\"ornblad, Department of Mathematics, University of Uppsala, Box 480, S-75106 Uppsala, Sweden}
 \email{erik.thornblad@math.uu.se}


\date{\today}

\begin{document}

\begin{abstract}
 We study the limiting degree distribution of the vertex splitting model introduced in \cite{DDJS:2009}. This is a model of randomly growing ordered trees, where in each time step the tree is separated into two components by splitting a vertex into two, and then inserting an edge between the two new vertices. Under some assumptions on the parameters, related to the growth of the maximal degree of the tree, we prove that the vertex degree densities converge almost surely to constants which satisfy a system of equations. Using this we are also able to strengthen and prove some previously non--rigorous results mentioned in the literature.
\end{abstract}
\maketitle

\section {Introduction}
The vertex splitting model is a recent model of randomly growing ordered trees introduced in \cite{DDJS:2009}. It is a modification of a model of randomly growing trees encountered in the theory of RNA-folding \cite{RNA}. The parameters of the model are non-negative weights $(w_{i,j})_{i,j\geq 0}$, symmetric in the indices $i$ and $j$,  and the trees are grown randomly in discrete time steps according to the following rules: Let $T$ be an ordered tree, $V_T$ its set of vertices and denote the degree of a vertex $v$ by $\deg(v)$.
\begin {enumerate}
 \item Start with some finite tree $T_0$ at time $t_0 :=|V_{T_0}|$.
 \item Given a tree $T$ at time $t \geq t_0$, select a vertex $v$ in $T$ with probability
 \begin {align}
  \frac{w_{\deg(v)}}{\sum_{v'\in V_T} w_{\deg(v')}}
 \end {align}
 where
 \begin{align} \label{splittingweights}
 w_i:=\frac{i}{2}\sum_{j=1}^{i+1}w_{j,i+2-j}.
\end{align}
\item Partition the edges which contain $v$ into disjoint sets of adjacent edges: $E'$ of size $k-1$ and $E''$ of size $\deg(v)-k+1$, with probability $\frac{w_{k,\deg(v)+2-k}}{w_{\deg(v)}}$. 
\item Remove the vertex $v$ and the edges containing it and insert two new vertices $v'$ and $v''$. Connect $v'$ by an edge to all vertices $u$ such that $uv$ is an edge in $E'$ and  connect $v''$ to all vertices $w$ such that $wv$ is an edge in $E''$. Add the edge $v'v''$ (see Fig~\ref{f:split}).
\end {enumerate}

\begin{figure}[h!]
 \includegraphics[width=0.6\textwidth]{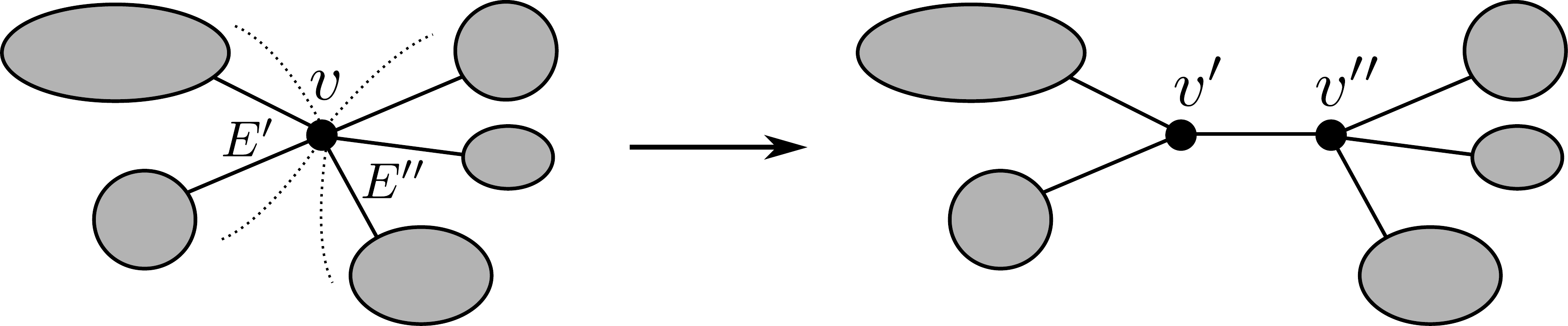}
 \caption{The vertex $v$ of degree $i := \deg(v)$ is split into two new vertices $v'$ and $v''$ of degrees $k := \deg(v')$ and $\ell := \deg(v'') = i+2-k$ respectively with probability $w_{k,\ell}/w_i$. Here $i=5$, $k=3$ and $\ell = 4$.} \label{f:split}
\end{figure}

\begin {remark}
 The numbers $(w_i)_{i\geq 1}$ defined by  \eqref{splittingweights} are called \emph{splitting weights} and the numbers $(w_{i,j})_{i,j\geq 1}$ are referred to as \emph{partitioning weights}. We say that the vertex $v$ which is selected in step (2) is \emph{split} into the vertices $v'$ and $v''$ in step (4). 
 
 In step (3) of the growth rules, the adjacency of edges is well-defined since the trees are ordered.  There are in general many choices of the sets $E'$ and $E''$. When $\deg(v)$ is even and $k-1 = \deg(v)/2$ there are exactly $\deg(v)/2$ different choices but otherwise there are $\deg(v)$ different choices.
\end {remark}
We are interested in studying the distribution of vertex degrees in large random trees $T$ grown according to the  above rules. More precisely, let $n_{t,k}$ be the number of vertices of degree $k$ in the tree at time $t$. In \cite{DDJS:2009} the asymptotics of the expected values $\E(n_{t,k})$ were studied under the following assumptions.
\begin {enumerate}
 \item [(A1)] The splitting weights are linear, i.e.~satisfy $w_i = ai+b$ for some real numbers $a$ and $b$ such that $w_i\geq 0$ for all $i\geq 1$.
 \item [(A2)] There is a finite integer $\dmax$ such that $w_{j,k} = 0$ if either $j$ or $k$ exceeds $\dmax$ (no vertices of degree greater than $\dmax$ are created in the growth process) and $w_{1,k} = w_{k,1}>0$ for all $2 \leq k\leq \dmax$ (it is possible to create vertices of degree $\dmax$ starting from any initial tree). (Corresponds to (1) in Lemma 2.3 in \cite{DDJS:2009}.)
 \item [(A3)] $w_{i,\dmax+2-i}>0$ for some $i$ satisfying $2\leq i \leq \dmax-1$ (it is possible to split  vertices of degree $\dmax$). (Corresponds to (2) in Lemma 2.3 in \cite{DDJS:2009}.)
 \item [(A4)] The $\dmax \times \dmax$ matrix $B$ given by the matrix elements
 \begin {align}
B_{ij} = j w_{i,j+2-i}-\delta_{ij} w_i, \quad 1\leq i,j\leq \dmax
 \end {align}
is diagonalizable. (Appears in Theorem 2.5 in \cite{DDJS:2009}.)
\end {enumerate}
It was shown that under these assumptions the limits $\rho_k := \lim_{t\rightarrow\infty} \E(n_{t,k})/t$ exist for $1\leq k \leq \dmax$ and are the unique positive solutions to 
\begin {align}\label{eq:solution}
 \rho_k = -\frac{w_k}{w_2}\rho_k  + \sum_{i=k-1}^{\dmax} i\frac{w_{k,i+2-k}}{w_2}\rho_i
\end {align} such that $\sum_{k=1}^{\infty}\rho_k = 1$ and $\sum_{k=1}^{\infty}k\rho_k = 2$. The last two should be compared to the equations
\begin{align}\label{eq:sum}
 \sum_{i=1}^{\dmax}n_{t,i}=t \quad \text{ and } \quad \sum_{i=1}^{\dmax}in_{t,i}=2t-2,
\end{align}
which is (2.7) in \cite{DDJS:2009}. 

The condition (A1) is a very convenient technical condition and we will assume that it holds throughout the paper. The reason is that for linear splitting weights, the total weight of selecting a vertex in $T$ only depends on the number of vertices in $T$, namely
\begin{align} \label{eq:wt}
\sum_{v \in V_T} w_{\deg(v)} = \sum_{i=1}^{\dmax} w_i n_{t,i} = (2a+b)t - 2a = w_2 t - 2a =: W_t
\end{align}
by \eqref{eq:sum}.

In this paper we prove stronger results concerning convergence of the random variables $n_{t,k}$. First of all, we prove almost sure convergence of $n_{t,k}/t$ towards $\rho_k$ satisfying \eqref{eq:solution} which immediately implies, by the boundedness of $n_{t,k}/t$  and the dominated convergence theorem, that the expected value converges. Furthermore, we relax some of the conditions (A2)--(A4) stated above as will be mentioned in the statement of the results in Theorem \ref{thm:as}. In particular we do not require the matrix $B$ to be diagonalizable when $\dmax < \infty$ and we obtain partial results when there is no bound on the maximum degree.	  

Throughout this paper we do not take into account the structural properties of the trees, but only analyze the asymptotic vertex degree densities. This means that the analysis fits into the framework of generalized P\'{o}lya urn models. Consider $\dmax$ urns, labelled $1,2,\dots, \dmax$, initially containing some number of balls. At each time step, draw a ball from urn $i$ with probability proportional to $w_in_i$, where $n_i$ is the number of balls in the $i$:th urn. Then put two balls back in, one in the $k$:th urn and one in the $(i+2-k)$:th urn, where the pair $(k,i+2-k)$ is chosen with probability $w_{k,i+2-k}/w_i$. In this framework each vertex of degree $i$ in the vertex splitting tree corresponds to a ball in the $i$:th urn. The body of literature regarding P\'{o}lya urns is large. In particular, whenever $\dmax<\infty$, our results follow from \cite{Janson2004}. In fact stronger results regarding asymptotic joint normality are attainable in this regime, but we do not pursue this matter further. Even though the results of \cite{Janson2004} holds only for a finite number of urns, sometimes a reduction from a situation from infinitely many urns to a situation with finitely many urns is possible. For the vertex splitting model, such a reduction unfortunately only works for the subclass of splitting trees for which, for all large enough $i$, the only positive partitioning weights are $w_{1,i+1},w_{2,i}$.

By choosing the splitting and partitioning weights appropriately, the vertex splitting model contains several other known models. For instance, by letting $w_{1,i+1}$ be the only positive partitioning weight, we may retrieve the random recursive trees by choosing $w_i=1$ and the random plane recursive trees by choosing $w_i=i$. These were analyzed in \cite{Janson2005b} by using the aforementioned connection to generalized P\'{o}lya urns.

Without giving complete details, we also extend our results to the setting considered in the motivating paper \cite{RNA}, in which a correspondence between coloured splitting trees and arch deposition models is exploited to analyze the secondary structure of RNA folding. In this case each vertex is coloured either black or white. If a black vertex is chosen, it is recoloured white. If a white vertex is chosen, it splits (similar to the aforementioned 1--coloured case) into two black vertices. By and large the same methods apply, and moreover it turns out that there is a clear correspondence between the 1--coloured version and the 2--coloured version.

\subsection{Main results}
In the following we let $\dmax$ be a positive integer or infinite. If $\dmax < \infty$ we will assume that initially no vertex has degree $> \dmax$ and that Condition (A2) above is satisfied. Let $s:=\inf\{iw_{1,{i+1}} \ : \  1 \leq i < \dmax\}$. For each $k$ such that $1 \leq k < \dmax +1$, define the sequence $(a_k^{(j)})_{j\geq 0}$ as follows. For $k=1$ let
\begin{align}
 a_1^{(0)}&=0, \label{eq:seq10} \\
 a_1^{(j+1)}&=\frac{1}{w_2+s}\left(s+\sum_{i = 2}^{\infty}(iw_{1,i+1}-s)a_i^{(j)} \right), \label{eq:seq1j}
\end{align}
and when $2 \leq k < \dmax + 1$ let
\begin{align}
 a_k^{(0)}&=0, \label{eq:seqk0} \\
 a_k^{(j+1)}&=\frac{1}{w_2+w_k}\sum_{i = k-1}^{\infty}iw_{k,i-k+2}a_{i}^{(j)}.\label{eq:seqkj} 
\end{align} 
The main result of the paper is the following.

\begin{theorem}\label{thm:as}
Suppose that $\inf\{iw_{1,{i+1}} \ : \  1 \leq i < \dmax\}>0$. Then for each $k$ such that $1 \leq k < \dmax +1$ the following limits exist and it holds almost surely that
\begin{align}
 \lim_{t\to \infty}\frac{n_{t,k}}{t} = \lim_{j\to \infty} a_k^{(j)}=: a_k
\end{align}
and $(a_k)_{1 \leq k <  \dmax +1}$ is a positive solution to 
\begin{align}
 a_k&= - \frac{w_k}{w_2} a_k + \sum_{i= k-1}^{\infty}i \frac{w_{k,i+2-k}}{w_2}a_i \qquad \qquad \qquad \qquad (k\geq 1). \label{eq:ask}
\end{align}
satisfying
\begin {align} \label{eq:sums}
 \sum_{k= 1}^{\infty} a_k = 1 \quad \text{and} \quad \sum_{k= 1}^{\infty} k a_k = 2.
\end {align}
When $\dmax < \infty$ the solution is unique.
\end{theorem}

\begin {remark}\label{rem:rem1}
We note that Theorem \ref{thm:as} requires fewer assumptions than those in (A1)--(A4). However, we have not provided conditions which guarantee that $(a_k)_{k=1}^{\infty}$ is a unique positive solution to \eqref{eq:ask} satisfying \eqref{eq:sums}.Whenever $\dmax<\infty$ however, uniqueness is guaranteed since the condition that $\inf\{iw_{1,{i+1}} \ : 1 \leq i < \dmax\}>0$ is  equivalent to Condition (A2) above which guarantees that  \eqref{eq:solution} has rank at least $\dmax-1$. Equations \eqref{eq:sums} fix the remaining constant. In the case $\dmax=\infty$ we do not know how to prove that the system has a unique solution, but we comment on some special examples in the next section.

 The case $\dmax<\infty$ is implicitly covered by the case $\dmax=\infty$ by assuming that there is some $i$ for which $w_{1,i+1}=0$. This allows for the identification of three regimes as follows.
\begin{enumerate}[I.]
 \item There is some $i\geq 1$ such that $w_{1,i+1}=0$.
 \item There is no $i\geq 1$ such that $w_{1,i+1}=0$, and $\inf\{iw_{1,{i+1}} \ : 1 \leq i < \dmax\}=0$.
 \item It holds that $\inf\{iw_{1,{i+1}} \ : 1 \leq i < \dmax\}>0$.
\end{enumerate}
 Theorem \ref{thm:as} deals with Case I and III, but we do not know how to extend these results to case II. The assumption $\inf\{iw_{1,{i+1}} \ : 1 \leq i < \dmax\}>0$ in some sense ensures that the probability of performing the split $i\mapsto (1,i+1)$ does not become too small as $i$ grows large.
\end{remark}

The rest of the paper is outlined as follows. In Section 2 we give some examples of the case $\dmax = \infty$ to which Theorem \ref{thm:as} is applicable and where uniqueness of the solution to \eqref{eq:ask}--\eqref{eq:sums} is  proved. The proof of Theorem \ref{thm:as} along with some technical results is presented in Section 3. Finally, in Section 4 we sketch similar results for a two--coloured vertex splitting model that was originally considered in \cite{RNA}.

\section{Examples}
Before we turn to the proof of Theorem \ref{thm:as}, we consider some explicit examples that can be analysed using Theorem \ref{thm:as}. Throughout we assume that $w_i=ai+b$ where $ai+b>0$ for all $2\leq i \leq \dmax-1$. If $a\neq 0$, it is easy to see that the growth rules defined in the introduction are not changed if we take $w_i=i+x$ with $x=b/a$, so for notational convenience we will sometimes use this definition instead, unless we are interested in the case of constant splitting weights.

We will focus on the case where there is no bound on the vertex degrees since the case when $\dmax < \infty$ is already covered. In general we cannot show that the system of equations in Theorem \ref{thm:as} has a unique solution. However, in some specific cases this is possible, and we present some of these here. 

\subsection{Preferential attachment} \label{sec:prefatt}
Preferential attachment--type models are obtained by attaching new edges to existing vertices, where the vertex is chosen with probability proportional to the vertex degree. In the vertex--splitting model with $\dmax=\infty$, this is obtained by setting all partitioning weights to zero, except for $w_{i+1,1}=w_{1,i+1}=\frac{w_i}{i}$ for all $i\geq 1$. The conditions of Theorem \ref{thm:as} are satisfied, so the limiting degree distribution $(a_i)_{i=1}^{\infty}$ satisfies the equations
\begin{align}
 (w_1+w_2)a_1&=\sum_{i=1}^{\infty}w_ia_i  \label{eq:prefatt1} \\
(w_k+w_2)a_k&=w_{k-1}a_{k-1}, \qquad k\geq 2.\label{eq:prefattk}
\end{align}
The relations $\sum_{i=1}^{\infty}a_i=1$ and $\sum_{i=1}^{\infty}ia_i=2$ imply that $\sum_{i=1}^{\infty}w_ia_i=w_2$, so $a_1=\frac{w_2}{w_1+w_2}$. Then, iterating (\ref{eq:prefattk}) we find that
\begin{align}
 a_k = \frac{w_{k-1}}{w_k+w_2}a_{k-1} = \cdots = a_1\prod_{i=2}^{k} \frac{w_{i-1}}{w_{i}+w_2} = \frac{w_2}{w_k} \prod_{i=1}^k \frac{w_{i}}{w_{i}+w_2},
\end{align}
for all $k\geq 1$. It is not difficult to analyse the asymptotics of the sequence $(a_k)_{k=1}^{\infty}$. One can show that when $a \neq 0$
\begin{align} \label{eq:pa0}
 a_k = \frac{(2+x)\Gamma\left(2x+3\right)\Gamma\left(k+x+1\right)}{(k+x)\Gamma\left(x+1\right)\Gamma\left(k+2x+3\right)} \sim  \frac{(2+x)\Gamma\left(2x+3\right)}{\Gamma\left(x+1\right)} k^{-3-x} \quad \text{as } k\to \infty
\end{align}
by using standard asymptotics for the Gamma function.

If $w_i=i$, then the preferential attachment model is equivalent to the model of random plane recursive trees. In this case the exact solution is given by 
\begin{align} \label{eq:pa1}
 a_k = \frac{4}{k(k+1)(k+2)}, \qquad k\geq 1.
\end{align}
This was originally proved by M\'{o}ri \cite{MoriRandomTrees} using martingale methods. He also achieved results on joint normality of the degree densities. The case $w_i=1$ is equivalent to the case of random recursive trees; the exact solution being
\begin {align}\label{eq:pa2}
 a_k = 2^{-k}, \qquad k\geq 1,
\end {align}
in this case. Using the connection to generalized P\'{o}lya urns mentioned in the introduction, Janson \cite{Janson2005b} proved stronger results that imply the almost sure convergence to the densities in \eqref{eq:pa1} and \eqref{eq:pa2}.

\subsection{Uniform partitioning weights} \label{sec:uniform}
 Let $w_i=i+x$, where $x>-1$. In \cite{DDJS:2009}, the expected degree densities were studied in the case of \emph{uniform} partitioning weights, i.e.
\begin{align}
 w_{i,k+2-i} = w_k / \binom{k+1}{2} = \frac{2 w_k}{k(k+1)}.
\end{align}
The methods in \cite{DDJS:2009} were non-rigorous in this case but the results were correct as we confirm here.
By Theorem \ref{thm:as} the asymptotic vertex degree densities $(a_i)_{i=1}^{\infty}$ satisfy
\begin{align}
 (w_2+w_k)a_k=\sum_{i=k-1}^{\infty}\frac{2}{i+1}w_ia_i,\qquad \qquad  k \geq 1.
\end{align}
Subtracting the $k$:th equation from the $(k+1)$:st yields the recursion
\begin{align} \label{eq:unif1}
 (w_2+w_k)a_k-(w_2+w_{k+1})a_{k+1}&=\frac{2}{k}w_{k-1}a_{k-1},  \qquad \qquad  k \geq 1
\end{align}
where we have defined $a_0=0$. The solution is given by
\begin{align}\label{eq:unif2}
 a_k=\frac{1}{C(x)} \frac{2^{k-1} \Gamma(k+x)}{\Gamma(k)\Gamma(k+3+2x)}(k+1+2x),
\end{align}
where
\begin{align}
 C(x)=\frac{e\sqrt{\pi}2^{-\frac{3}{2}-x}I_{\frac{1}{2}+x}(1)}{2+x}
\end{align}
and $I_{\frac{1}{2}+x}$ is the modified Bessel function of the first kind. In particular, for $x=0$, i.e. splitting weights $w_k=k$, we have
\begin{align}
 a_k = \frac{2^{k+2}(k+1)}{(e^2-1)(k+2)!}.
\end{align}
It is easy to see that the obtained solutions are in fact unique. Namely, from \eqref{eq:unif1} one can inductively determine constants $C_1,C_2\dots$ such that $a_k=C_ka_1$ for all $k\geq 1$. The condition $\sum_{i=1}^{\infty}a_i=1$ determines $a_1$ uniquely as a function of these constants.

If the splitting weights are constant, i.e. $w_i=b$ for some $b>0$, then the solution to \eqref{eq:unif1} is given by
\begin{align}
 a_k=\frac{1}{e}\frac{1}{(k-1)!}.
\end{align}
However, in this case $iw_{1,i+1}=O(i^{-1})$, so the conditions of Theorem \ref{thm:as} are not satisfied. This falls into Case II identified in Remark \ref{rem:rem1}, so it does not follow by our results that these are also the almost sure limiting vertex densities.

\subsection{An infinite class}

Our next example provides an infinite class of splitting trees  for which a unique solution to the system of equations is attainable. It includes the preferential attachment model and more generally a model of trees which grow by \emph{attachment and grafting}, studied in \cite{Siggi2012} (see below).

Let $(\alpha_i)_{i=1}^{\infty}$ be a sequence in $(0,1]$, such that $\inf_{i\geq 1} w_i \alpha_i > 0$. Suppose there exists some $M\geq 2$  such that $iw_{1,i+1}=\alpha_{i}w_i$, $iw_{2,i}=(1-\alpha_i)w_i$ for all $i\geq M$ with the exception that $w_{2,2} = (1-\alpha_2)w_2$ when $M=2$. 

The conditions of Theorem \ref{thm:as} are satisfied. In particular, for $k>M$ the asymptotic degree sequence satisfies
\begin{align}\label{eq:infrec}
 (w_k+w_2)a_k&=\alpha_{k-1}w_{k-1}a_{k-1}+(1-\alpha_k)w_ka_{k} 
\end{align}
which is similar to the expressions in Section \ref{sec:prefatt}. Iterating we find 

\begin{align}
a_k=a_M\prod_{i=M+1}^{k}\frac{\alpha_{i-1}w_{i-1}}{w_2+\alpha_iw_i}=:C_ka_M
\end{align}
for all $k>M$. 

For any $1\leq k \leq M$  
\begin{align}
\begin{split}
 (w_k+w_2)a_k = \sum_{i=k-1}^{\infty}iw_{k,i-k+2}a_i 
 = \sum_{i=k-1}^{M}iw_{k,i-k+2}a_i +a_M\sum_{i=M+1}^{\infty}iw_{k,i-k+2}C_i
\end{split}
\end{align}
where the final sum is zero unless $k=1$ or $k=2$. For instance, choosing $k=M$ yields 
\begin{align}
 (w_M+w_2)a_M=(M-1)w_{M,1}a_{M-1}+Mw_{M,2}a_M
\end{align}
so we can determine $C_{M-1}$ so that $a_{M-1}=C_{M-1}a_M$. Continue inductively to find a sequence $(C_i)_{i=1}^{M-1}$ such that $a_k=C_ka_M$ for $1\leq k<M$. In the case $k=1$ and $k=2$ one gets a system of two equations involving $a_1$ and $a_2$ which may easily be seen to have a unique solution which is a multiple of $a_M$. Now, the condition $\sum_{i=1}^{\infty}a_i=1$ means that $a_M\sum_{i=1}^{\infty}C_i=1$, so $a_M=\left(\sum_{i=1}^{\infty}C_i \right)^{-1}$, where we have put $C_M=1$ for consistency. But this allows us to uniquely determine the entire sequence $(a_k)_{k=1}^{\infty}$.

An example of a family of weights which belongs to the above class appears in \cite{Siggi2012} (with a minor modification in the dynamics which does not affect the limiting densities). The weights in \cite{Siggi2012} are defined in terms of two parameters $\alpha,\gamma \in [0,1]$ by 
\begin{align}
w_i = \left(\frac{\alpha}{2}+1-\gamma\right)i + 2\gamma - \alpha -1, \quad\text{for $i\geq 1$}
\end{align}
and by choosing $M=2$ and $\alpha_i = 1-\frac{\alpha i}{2w_i}$ for all $i\geq M$. Note that when $\alpha = 0$ this is the preferential attachment model. In \cite{Siggi2012} the solution to \eqref{eq:ask} when $0<\gamma<1$ was found to be
\begin {align}
a_1 &= \frac{1-\alpha}{1+\gamma-\alpha}\quad  \text{and}\\
a_k &= \frac{\gamma \Gamma\left(\frac{3-\alpha-\gamma}{1-\gamma}\right)\Gamma\left(k-2+\frac{1-\alpha}{1-\gamma}\right)}{(1+\gamma-\alpha)(2-\alpha)\Gamma\left(\frac{1-\alpha}{1-\gamma}\right)\Gamma\left(k-1+\frac{2-\alpha}{1-\gamma}\right)} \quad \text{for $k\geq 2$}
\end {align}
and for $\gamma = 1$
\begin {align}
 a_1 &= \frac{1-\alpha}{2-\alpha} \quad \text{and} \\
 a_k &= \frac{1}{(2-\alpha)^2} \left(\frac{1-\alpha}{2-\alpha}\right)^{k-2} \quad \text{for $k\geq 2$}.
\end {align}
The results agree with \eqref{eq:pa0} and \eqref{eq:pa2} when $\alpha = 0$ (preferential attachment). There was no proof in \cite{Siggi2012} that these solutions are the limiting degree densities but our Theorem \ref{thm:as} along with the uniqueness of the solution confirms that they are the almost sure limit. In the case $0<\gamma<1$, standard asymptotics of the Gamma function yield the power law
\begin {align}
 a_k \sim \frac{\gamma \Gamma\left(\frac{3-\alpha-\gamma}{1-\gamma}\right)}{(1+\gamma-\alpha)(2-\alpha)\Gamma\left(\frac{1-\alpha}{1-\gamma}\right)} k^{-\frac{2-\gamma}{1-\gamma}} \quad \text{as $k\to \infty$}
\end {align}
and when $\gamma = 1$ the densities decay exponentially with rate $(1-\alpha)/(2-\alpha)$.

\section{Proof of Theorem \ref{thm:as}}

We now turn to the proof of Theorem \ref{thm:as}. To simplify notation, we deal only with the case $\dmax = \infty$, the discussion being even simpler when $\dmax$ is finite. First we state a key lemma, that appears in slightly more general form in \cite{BackhauszMoriLemma}.
\begin{lemma}[Backhausz, Móri \cite{BackhauszMoriLemma}]
 Let $(\mathcal{F}_t)_{t=0}^{\infty}$ be a filtration. Let $(\xi_t)_{t=0}^{\infty}$ be a non-negative process adapted to $(\mathcal{F}_t)_{t=0}^{\infty}$, and let $(u_t)_{t=1}^{\infty},(v_t)_{t=1}^{\infty}$ be non--negative predictable processes such that $u_t<t$ for all $t\geq 1$ and $\lim_{t\to \infty} u_t = u>0$ exists almost surely. Let $w$ be a positive constant. Suppose that there exists $\delta >0$ such that $\mathbb{E}[(\xi_t-\xi_{t-1})^2 | \mathcal{F}_t]=O(t^{1-\delta})$. If 
\begin{align}
    \liminf_{t\to \infty}\frac{v_t}{w}\geq v
   \end{align} for some constant $v\geq 0$ and 
\begin{align}
 \mathbb{E}[\xi_t | \mathcal{F}_{t-1}]\geq \left(1-\frac{u_t}{t} \right)\xi_{t-1}+v_t,
\end{align}
then
\begin{align}
 \liminf_{t\to \infty}\frac{\xi_t}{tw}\geq \frac{v}{u+1} \quad \text{ a.s.}
\end{align}
\label{lem:backhauszmori}
\end{lemma}

We use Lemma \ref{lem:backhauszmori} to prove the following lemma. This essentially follows the approach taken in the papers \cite{BackhauszMori,Thornblad2015}. We note here that the $s$--term present below does not appear in Theorem \ref{thm:as} since we shall later prove that $\sum_{i=1}^{\infty}a_i=1$, but this is a priori not known.
\begin{lemma}\label{thm:liminf} 
 Suppose that $s:=\inf\{iw_{1,{i+1}} \ : \  i \geq 1\}>0$. 
Then  $\lim_{j\to \infty}a_k^{(j)}=:a_k$ exist for each $k$, $(a_k)_{ k \geq 1}$ is a positive bounded sequence satisfying
\begin{align}
 (w_2+w_1)a_1&=\sum_{i=1}^{\infty}iw_{1,i+1}a_i+s\left(1-\sum_{i=1}^{\infty}a_i\right),\label{eq:liminf1}  \\
 (w_2+w_k)a_k&=\sum_{i=k-1}^{\infty}iw_{k,i+2-k}a_i \qquad \qquad \qquad \qquad (k\geq 2).\label{eq:liminfk}
\end{align}
and $\liminf_{t\to \infty}\frac{n_{t,k}}{t} \geq a_k$ holds almost surely,
\end{lemma}
Let us first state the idea behind the proof of Lemma \ref{thm:liminf}. To ease notation we define $A_k=\liminf_{t \to \infty}\frac{n_{t,k}}{t}$ for each $k\geq 1$.
\begin{enumerate}
 \item For each $k\geq 1$, show by induction that $a_k^{(j)}\leq A_k$ for all $j\geq 0$. 
\item Prove that $(a_k^{(j)})_{j=1}^{\infty}$ is monotonically increasing (in $j$) for each $k$. Since each such sequence lies in the bounded set $[0,1]$, the limit $\lim_{j\to \infty}a_k^{(j)}=a_k$ exists. Then we have that $a_k\leq A_k$.
\end{enumerate}

\begin{proof}
 Recall the definition of the total weight $W_t$ in \eqref{eq:wt}. The following expressions for the expected number of vertices of degree $k$, conditional on the tree at the previous time step, easily follow from the growth rules. For $k=1$ we have that 
\begin{align}
\begin{split}
 \mathbb{E}[n_{t,1} \ | \ \mathcal{F}_{t-1}] 
 \ &=\  n_{t-1,1}+\frac{1}{W_{t-1}}\sum_{i=2}^{\infty}iw_{1,i+1}n_{t-1,i} \\
 \ &=\  n_{t-1,1}\left(1-\frac{s}{W_{t-1}} \right)+\frac{s}{W_{t-1}}n_{t-1,1}+\frac{1}{W_{t-1}}\sum_{i=2}^{\infty}iw_{1,i+1}n_{t-1,i} \\
&  = \ n_{t-1,1}\left(1-\frac{s}{W_{t-1}} \right)+\frac{s}{W_{t-1}}(t-1)+\frac{1}{W_{t-1}}\sum_{i=2}^{\infty}(iw_{1,i+1}-s)n_{t-1,i}.
\end{split}
\end{align}
In the last line we used the fact that $n_{t-1,1}=(t-1)-\sum_{i=2}^{\infty}n_{t-1,i}$.

For $k\geq 2$ we have that
\begin{align}
\mathbb{E}[n_{t,k} \ | \ \mathcal{F}_{t-1}] 
 \ =\  n_{t-1,k}\left(1-\frac{w_k}{W_{t-1}} \right) 
 \ + \ \frac{1}{W_{t-1}}\sum_{i=k-1}^{\infty}iw_{k,i-k+2}n_{t-1,i} \label{cond}
\end{align}

We shall use the above analysis along with Lemma \ref{lem:backhauszmori}. By induction we prove that $a_k^{(j)}\leq A_k$ for all $j\geq 0$. For $j=0$ we clearly have $a_k^{(0)}=0\leq A_k$. Suppose that $a_k^{(j)}\leq A_k$ for some $j$ and for all $k\geq 1$. We prove first that $a_1^{(j+1)}\leq A_1$. For this, define the following variables:
\begin{align}
\begin{cases}
 \xi_t=n_{t,1}, \\ 
w=1, \\
u_t=\frac{s}{W_{t-1}}t, \\ 
v_t= \frac{s}{W_{t-1}}(t-1)+\frac{1}{W_{t-1}}\sum_{i=2}^{\infty}(iw_{1,i+1}-s)n_{t-1,i}.
\end{cases}
\end{align}
We note that $(u_t)_{t=1}^{\infty}$ and $(v_t)_{t=1}^{\infty}$ are positive predictable sequences and that $\xi_t$ is non--negative and adapted. Furthermore, we have that $u_t=\frac{s}{W_{t-1}}t\to \frac{s}{w_2}=:u$. The condition $u_t<t$ in Lemma \ref{lem:backhauszmori} is satisfied for large enough $t$, which is enough. In fact, the initial starting tree is irrelevant, so if $u_{t_0}\geq t_0$, one can grow the tree and wait until $u_t<t$ occurs, at which point Lemma \ref{lem:backhauszmori} can be applied. Recall that $s=\inf\{iw_{1,i+1} \ : \ i\geq 1 \}$. Using Fatou's lemma and the induction hypothesis we find that
\begin{align}
\begin{split}
 \liminf_{t\to \infty}v_t
&\geq \frac{s}{w_2}+\sum_{i=2}^{\infty}(iw_{1,i+1}-s)\liminf_{t\to \infty}\frac{n_{t-1,i}}{W_{t-1}} \\
&\geq \frac{s}{w_2}+\sum_{i=2}^{\infty}(iw_{1,i+1}-s)\frac{a_i^{(j)}}{w_2} \\
&=\frac{s}{w_2}+\frac{1}{w_2}\sum_{i=2}^{\infty}(iw_{1,i+1}-s)a_i^{(j)} \\
&:=v.
\end{split}
\end{align}
Applying Lemma \ref{lem:backhauszmori} we find that
\begin{align}
\begin{split}
 A_1=\liminf_{t\to \infty}\frac{n_{t,1}}{t} 
&\geq \frac{v}{u+1} \\
&=\frac{\frac{s}{w_2}+\frac{1}{w_2}\sum_{i=2}^{\infty}(iw_{1,i+1}-s)a_i^{(j)}}{\frac{s}{w_2}+1} \\
&=\frac{1}{w_2+s}\left(s+\sum_{i=2}^{\infty}(iw_{1,i+1}-s)a_i^{(j)}\right) \\
&=a_1^{(j+1)}.
\end{split}
\end{align}

For $k\geq 2$ define the following variables: 
\begin{align}
\begin{cases}
 \xi_t=n_{t,k}, \\ 
w=1, \\
u_t=\frac{w_k}{W_{t-1}}t, \\
v_t= \frac{1}{W_{t-1}}\sum_{i=k-1}^{\infty}iw_{k,i-k+2}n_{t-1,i}.
\end{cases}
\end{align} 
We note that $(u_t)_{t=1}^{\infty}$ and $(v_t)_{t=1}^{\infty}$ are positive predictable sequences and that $\xi_t$ is non--negative and adapted. Furthermore, we have that $u_t=\frac{w_k}{W_{t-1}}t\to \frac{w_k}{w_2}=:u$. We now apply Fatou's lemma and use the induction hypothesis and find that
\begin{align}
\begin{split}
 \liminf_{t\to \infty}v_t
&\geq \sum_{i=k-1}^{\infty}iw_{k,i-k+2}\liminf_{t\to \infty}\frac{n_{t-1,i}}{W_{t-1}} \\
&\geq \sum_{i=k-1}^{\infty}iw_{k,i-k+2}\frac{a_i^{(j)}}{w_2} \\
&=\frac{1}{w_2}\sum_{i=k-1}^{\infty}iw_{k,i-k+2}a_i^{(j)} \\
&:=v.
\end{split}
\end{align}
Applying Lemma \ref{lem:backhauszmori} we find that 
\begin{align}
\begin{split}
 A_k=\liminf_{t\to \infty}\frac{n_{t,k}}{t} 
&\geq \frac{v}{u+1} \\
&=\frac{\frac{1}{w_2}\sum_{i=k-1}^{\infty}iw_{k,i-k+2}a_i^{(j)}}{\frac{w_k}{w_2}+1} \\
&=\frac{1}{w_2+w_k}\sum_{i=k-1}^{\infty}iw_{k,i-k+2}a_i^{(j)} \\
&=a_1^{(j+1)}.
\end{split}
\end{align}

We note finally that the technical condition $\mathbb{E}[(\xi_t-\xi_{t-1})^2 | \mathcal{F}_t]=O(t^{1-\delta})$ is satisfied for all $k\geq 1$. Indeed, the conditional expectation $\mathbb{E}[(\xi_t-\xi_{t-1})^2 | \mathcal{F}_t]$ is bounded above by $4$. This completes the Step 1, i.e. we have showed that $a_k^{(j)}\leq A_k$ for all $j\geq 0$ and all $k\geq 1$. 

We now prove that for each $k\geq 1$, the sequence $(a_k^{(j)})_{j=0}^{\infty}$ is increasing. We use an inductive argument. By construction we have that $a_1^{(0)}=0\leq \frac{s}{s+w_2}= a_1^{(1)}$. Suppose that the statement is true for some $j$. Recall that $s:=\inf\{iw_{1,{i+1}} \ : \ i\geq 1\}>0$, so in particular $iw_{1,i+1}-s\geq 0$ for all $i\geq 1$. Then
\begin{align}
\begin{split}
  a_1^{(j+1)}&=\frac{1}{w_2+s}\left(s+\sum_{i=2}^{\infty}(iw_{1,i+1}-s)a_i^{(j)} \right) \\
&\geq \frac{1}{w_2+s}\left(s+\sum_{i=2}^{\infty}(iw_{1,i+1}-s)a_i^{(j-1)} \right) \\
&\geq a_1^{(j)}.
\end{split}
\label{eq:increasing}
\end{align}
This proves that the sequence  $(a_k^{(j)})_{j=0}^{\infty}$ is increasing for $k=1$. The proof for $k\geq 2$ is similar and we omit this.

For each $k\geq 1$ we thus have that the sequence $(a_k^{(j)})_{j=0}^{\infty}$ is an increasing sequence. It is bounded above by $A_k\leq 1$, so each such sequence must be convergent and we have the existence of a limit 
\begin{align}
 \lim_{j\to \infty}a_k^{(j)}=a_k.
\end{align}

By taking limits in (\ref{eq:seq10}--\ref{eq:seqkj}) we find that
\begin{align}
  a_1&=\frac{1}{w_2+s}\left(s+\sum_{i=2}^{\infty}(iw_{1,i+1}-s)a_i \right) \label{eq:lim1j}
\end{align}
and for each $k\geq 2$
\begin{align}
 a_k&=\frac{1}{w_2+w_k}\sum_{i=k-1}^{\infty}iw_{k,i-k+2}a_{i}.\label{eq:limkj} 
\end{align}
Interchanging limits and summation is justified since all terms are positive. Using $w_1=w_{1,2}$, it is now easy to see that (\ref{eq:lim1j}) is equivalent to
\begin{align}
 (w_2+w_1)a_1=\sum_{i=1}^{\infty}iw_{1,i}a_i+s\left(1-\sum_{i=1}^{\infty}a_i \right).
\end{align}
This completes the proof.
\end{proof}

Next we show that the sequence $(a_k)_{k=1}^{\infty}$, constructed in Lemma \ref{thm:liminf}, defines a probability distribution. After this we prove Theorem \ref{thm:as}.  First we prove the following lemma, which is in the spirit of Lemma 2.3 in \cite{DDJS:2009} 
\begin{lemma}\label{lem:eigeneqs}
 The sequence $(a_k)_{k=1}^{\infty}$ satisfies
$\sum_{i= 1}^{\infty}a_i=1$ and $\sum_{i= 1}^{\infty}ia_i=2$.
\end{lemma}

\begin{proof}
Recall from Lemma \ref{thm:liminf} that for each $j$ and $k$, $a_k^{(j)} \leq \liminf_{n\to \infty}\frac{n_{t,k}}{t}$. Thus, for each $j$,
\begin {align}
 \sum_{k= 1}^{\infty} a_k^{(j)} \leq \sum_{k= 1}^{\infty} \liminf_{t\to \infty}\frac{n_{t,k}}{t} \leq \liminf_{t\to \infty} \frac{1}{t} \sum_{k=1}^{\infty}n_{t,k} = 1
\end {align}
by Fatou's lemma and (\ref{eq:sum}). Similarly, for each $j$, $\sum_{k= 1}^{\infty} k a_k^{(j)} \leq 2$. Letting $j\to \infty$ it follows from monotonicity of $a_k^{(j)}$ that the following series are convergent and satisfy
\begin {align} \label{eq:convsums}
 \sum_{k= 1}^{\infty} a_k \leq 1 \quad \text{and} \quad \sum_{k= 1}^{\infty} k a_k \leq 2.
\end {align}

Summing (\ref{eq:liminf1}) and (\ref{eq:liminfk}) over $k=1,2,\dots$ we find that
\begin{align}
\begin{split} \label{eq:A}
 w_2\sum_{k=1}^{\infty}a_k
&=-\sum_{k=1}^{\infty}w_ka_k+\sum_{k=1}^{\infty}\sum_{i=k-1}^{\infty}iw_{k,i-k+2}a_i+s\left(1-\sum_{i=1}^{\infty}a_i\right) \\
&=-\sum_{k=1}^{\infty}w_ka_k+2\sum_{i=1}^{\infty}w_ia_i+s\left(1-\sum_{i=1}^{\infty}a_i\right) \\
&=\sum_{i=1}^{\infty}w_ia_i+s\left(1-\sum_{i=1}^{\infty}a_i\right)
\end{split}
\end{align}
Similarly, multiplying the $k$:th equation by $k$ and summing over $k=1,2,\dots, N$ we find by swapping sums and using 
\begin {align}
 \sum_{k=1}^{i+1}kw_{k,i-k+2} = \frac{i+2}{2}\sum_{k=1}^{i+1}w_{k,i-k+2}
\end {align}
that
\begin{align}
 \begin{split}
  w_2\sum_{k=1}^{N}ka_k
&=-\sum_{k=1}^{N}kw_ka_k+\sum_{k=1}^{N}k\sum_{i=k-1}^{\infty}iw_{k,i-k+2}a_i+s\left(1-\sum_{i=1}^{N}a_i\right) \\
&=-\sum_{k=1}^{N}kw_ka_k+\sum_{i=1}^{N-1}i \left( \sum_{k=1}^{i+1}kw_{k,i-k+2}\right)a_i+\sum_{i=N}^{\infty}i\sum_{k=1}^{N}kw_{k,i-k+2}a_i+s\left(1-\sum_{i=1}^{N}a_i\right) \\
&=-\sum_{k=1}^{N}kw_ka_k+\sum_{i=1}^{N-1}i\frac{i+2}{2}\sum_{k=1}^{i+1}w_{k,i-k+2}a_i+\sum_{i=N}^{\infty}i\sum_{k=1}^{N}kw_{k,i-k+2}a_i+s\left(1-\sum_{i=1}^{N}a_i\right) \\
&=-\sum_{k=1}^{N}kw_ka_k+\sum_{i=1}^{N-1}(i+2)w_ia_i+\sum_{i=N}^{\infty}i\sum_{k=1}^{N}kw_{k,i-k+2}a_i+s\left(1-\sum_{i=1}^{N}a_i\right) 
 \end{split}
\end{align}
which yields, with some simple rewriting
\begin{align}
 \begin{split}
  w_2\sum_{k=1}^{N}ka_k -2 \sum_{k=1}^{N-1}w_ka_k -s\left(1-\sum_{i=1}^{N}a_i\right)
&= - N w_N a_N + \sum_{i =N}^\infty i \sum_{k=1}^N k w_{k,i-k+2}a_i.
 \end{split}
\end{align}

The limit, as $N\to\infty$, of the left hand side exists by \eqref{eq:convsums} and thus the limit of the right hand side exists, denote it by  
\begin{align}
 x := \lim_{N\to \infty} \left(- N w_N a_N + \sum_{i =N}^\infty i \sum_{k=1}^N k w_{k,i-k+2}a_i\right).
\end{align}
Then
\begin{align} \label{eq:B}
   w_2\sum_{k=1}^{\infty}a_k -2 \sum_{k=1}^{\infty}w_ka_k -s\left(1-\sum_{i=1}^{\infty}a_i\right) = x.
\end{align}

Finally, let $A=\sum_{i=1}^{\infty}a_i$ and $B=\sum_{i=1}^{\infty}ia_i$. Putting $w_k=ak+b$ in \eqref{eq:A} and \eqref{eq:B} we obtain the linear system of equations
\begin{align}
\begin{cases}
  (2a+b)A=aB+bA+s-sA, \\
 (2a+b)B=2(aB+bA)+s-sA+x
\end{cases}
\end{align}
having solutions
\begin {align}
 A &= 1+\frac{ax}{s(a+b)} \\
 B &= 2+\frac{(2a+s)x}{s(a+b)}.
\end {align}
Since $\dmax = \infty$, $a\geq 0$ and thus by \eqref{eq:convsums} it necessarily holds that $x\leq 0$.
If $x<0$ then there is an $M$ such that for all $N\geq M$, $N w_N a_N > -x/2$. Therefore
\begin {align}
 \sum_{N=1}^\infty w_N a_N > \sum_{N = M}^{\infty} \frac{-x}{2N} = \infty
\end {align}
which contradicts \eqref{eq:convsums}. Thus $x=0$ which gives $A=1$ and $B=2$, as desired.
\end{proof}

\begin{proof}[Proof of Theorem \ref{thm:as}]
Already knowing that $\liminf_{t\to \infty}\frac{n_{t,k}}{t}\geq a_k$, the idea is to prove that $\limsup_{t\to \infty}\frac{n_{t,k}}{t}\leq a_k$ for all $k\geq 1$. By Lemma \ref{lem:eigeneqs} we have that $\sum_{k=1}^{\infty}a_k=1$. By definition it holds that $\sum_{k=1}^{\infty}\frac{n_{t,k}}{t}=1$.  The following calculation is routine and only uses Fatou's lemma and well--known facts about the limit inferior and limit superior. For any $k\geq 1$ we have that
\begin{align}
\begin{split}
 \limsup_{t\to \infty}\frac{n_{t,k}}{t}
=\limsup_{t\to \infty}\left(1-\sum_{\substack{j=1 \\ j\neq k}}^{\infty}\frac{n_{t,j}}{t} \right) 
&\leq 1-\sum_{\substack{j=1 \\ j\neq k}}^{\infty}\liminf_{t\to \infty}\frac{n_{t,j}}{t} \quad (\text{Fatou's Lemma})\\
&\leq 1-\sum_{\substack{j=1 \\ j\neq k}}^{\infty}a_j \quad (\text{Lemma.~\ref{thm:liminf}}) \\ 
&=a_k.
\end{split}
\end{align}
Thus, by the above along with Lemma \ref{thm:liminf}
\begin{align}\liminf_{t\to \infty}\frac{n_{t,k}}{t}\geq a_k\geq \limsup_{t\to \infty}\frac{n_{t,k}}{t}
 \end{align}
almost surely for all $k\geq 1$. This implies that for all $k\geq 1$
\begin{align}
 \lim_{t\to \infty}\frac{n_{t,k}}{t}=a_k
 \end{align} almost surely. Finally, Equations \eqref{eq:ask} and \eqref{eq:sums} follow from Lemma \ref{thm:liminf} and Lemma \ref{lem:eigeneqs}.
\end{proof}

\section{An extension to two colours}
As mentioned in the introduction, the motivation of this study originally comes from \cite{RNA}, wherein a two--coloured version of the vertex splitting tree was considered. In this model, each vertex in the tree is coloured either white or black, the number of black and white vertices of degree $i$ at time $t$ denoted by $n_{t,i}^{\black}$ and $n_{t,i}^{\circ}$ respectively. The parameters of the model are the \emph{splitting weights} $(w_i^{\black})_{i=1}^{\infty}, (w_i^{\circ})_{i=1}^{\infty}$ and symmetric \emph{partitioning} weights $(w_{i,j}^{\circ})_{i,j\geq 0}$, satisfying $w_i^{\circ}=\frac{i}{2}\sum_{i=1}^{j+1}w^{\circ}_{j,i+2-j}$. Start at time $t=2$ with a single edge with both its endpoints black. At each time step, the tree evolves as follows.
\begin{enumerate}
 \item Select a vertex $v$ in the tree with probability proportional to $w_{\deg(v)}^{\circ}$ if $v$ is white, and $w_{\deg(v)}^{\black}$ if $v$ is black.
  \item If $v$ is black, make it white. If $v$ is white, partition its edges into two disjoint sets of adjacent edges $E'$ of size $k-1$ and $E''$ of size $\deg(v)-k+1$ with probability $\frac{w^{\circ}_{k,\deg(v)+2-k}}{w_{\deg(v)}^{\circ}}$. Remove the vertex $v$ and its incident edges. Insert two new \emph{black} vertices $v'$ and $v''$, such that $v'$ is connected to all vertices $u$ such that $uv$ is an edge in $E'$, and $v''$ is connected to all vertices $w$ such that $wv$ is an edge in $E''$. Add the edge $v'v''$.
\end{enumerate}
The paper \cite{RNA} considered only the case corresponding to splitting weights $w_k^{\circ}=k+1$ and $w_{k}^{\black} = k$  (no bound on vertex degrees) and uniform partitioning weights 
\begin{align}
 w^{\circ}_{i,k+2-i} =  w_k^{\circ} / \binom{k+1}{2} & \quad \text{ for } i=1,\dots, k+1.         
\end{align} 
In this case the model is equivalent to a model of random RNA folding, see \cite{RNA} where the correspondence is explained in detail. 
Under the assumption that the limit exists, the authors in \cite{RNA} found that
\begin{align}
 \lim_{t\to \infty}\frac{\E[n_{t,k}^{\circ}]}{t} &= \frac{2^k k}{e^2(k+2)!}, \\
 \lim_{t\to \infty}\frac{\E[n_{t,k}^{\black}]}{t} &= \frac{2^k}{e^2(k+1)!},
\end{align}
for all $k\geq 1$.  We consider a wider class of splitting weights and partitioning weights and are able to prove an analogue of Theorem \ref{thm:as} replacing $\frac{\E[n_{t,k}^{\black}]}{t}$ and $\frac{\E[n_{t,k}^{\circ}]}{t}$ with $\frac{n_{t,k}^{\black}}{t}$ and $\frac{n_{t,k}^{\circ}}{t}$ respectively and  proving almost sure convergence. Thus we confirm, strengthen and generalize the above results from \cite{RNA}.

Let $w_k^{\circ}=\left(a-\frac{3}{2}b\right)k+a$ and $w_k^{\black}=\left(a-\frac{3}{2}b\right)k+b$. This choice ensures that the total weight grows linearly, i.e.
\begin{align}\label{eq:totweight}
 \sum_{i=1}^{\infty}\left(w_i^{\circ}n_{t,i}^{\circ}+w_i^{\black}n_{t,i}^{\black}\right) = (a-b)t+b
\end{align} 
which can be showed by induction. Moreover, up to a multiplicative constant this is the unique choice of splitting weights resulting in linear growth for the total weight. Note also that $a-b=w_2^{\black}/2=w_2^{\circ}/3$. It is also possible to show that
\begin{align}\label{eq:2totweight}
 \sum_{i=1}^{\infty} \left(3n_{t,i}^{\circ}+2n_{t,i}^{\black} \right)=t+2.
\end{align} 
Equations \eqref{eq:totweight} and \eqref{eq:2totweight} correspond to \eqref{eq:sum}.

In the following theorem, the sequences $(e_k^{\circ})_{k=1}^{\infty}$ and $(e_k^{\black})_{k=1}^{\infty}$ are constructed like the sequence $(a_k)_{k=1}^{\infty}$ in Theorem \ref{thm:as}, but we leave the exact details to the reader.

\begin{theorem}\label{thm:as2}
 Suppose that $\inf \{ iw^{\circ}_{1,i+1} \ : \ 1\leq i < \dmax + 1\}>0$. Then there exist two non--negative sequences $(e_k^{\circ})_{k=1}^{\infty}$ and $(e_k^{\black})_{k=1}^{\infty}$ such that
\begin{align}
& \lim_{t\to \infty} \frac{n_{t,k}^{\circ}}{t}\stackrel{a.s.}{=} e_k^{\circ} \\
&\lim_{t\to \infty} \frac{n_{t,k}^{\black}}{t}\stackrel{a.s.}{=}  e_k^{\black} 
\end{align}
satisfying
 \begin{align}
&(w_k^{\black}+w_2^{\black}/2)e_k^{\black} = \sum_{i=k-1}^{\infty} iw_{k,i-k+2} e_k^{\circ}, \\
&(w_k^{\circ}+w_2^{\circ}/3)e_k^{\circ}=w_k^{\black}e_k^{\black},
\end{align}
for all $1\leq k < \dmax +1$. Moreover $\sum_{i=1}^{\infty}(3e_i^{\circ}+2e_i^{\black})=1$ and $\sum_{i=1}^{\infty}(w_i^{\circ}e_{i}^{\circ}+w_i^{\black}e_i^{\black})=w_2^{\black}/2$.
\end{theorem}

Note that the quantities $\frac{n_{t,k}^{\circ}}{t}$ and $\frac{n_{t,k}^{\black}}{t}$ are not degree densities any more, since the tree does not grow whenever a black vertex is selected. The actual asymptotic degree densities are given by 
\begin{align}
\begin{split}
 \rho_i^{\black}
=\lim_{t\to \infty} \frac{n_{t,i}^{\black}}{\sum_{j=1}^{\infty}(n_{t,j}^{\black}+n_{t,j}^{\circ})} 
&=\lim_{t\to \infty}  \frac{n_{t,i}^{\black}}{t}\cdot \frac{t}{\sum_{j=1}^{\infty}(n_{t,j}^{\black}+n_{t,j}^{\circ})}\\
&=\frac{e_i^{\black}}{\sum_{j=1}^{\infty}(e_{j}^{\black}+e_{j}^{\circ})}
\end{split}
\end{align}
and
\begin{align}
  \rho_i^{\circ} = \frac{e_i^{\circ}}{\sum_{j=1}^{\infty}(e_{j}^{\black}+e_{j}^{\circ})}.
\end{align}

The proof of Theorem \ref{thm:as2} is similar to that of Theorem \ref{thm:as}, so we omit this. Instead we mention another result that relates the densities of any two--coloured process to a one--coloured process. Note however that this depends crucially on knowing that the solutions to the equations in Theorems \ref{thm:as} and \ref{thm:as2} are \emph{unique}, something we do not know in general. If this is known, the rest of the proof is straightforward and is omitted -- it amounts to showing that the appropriate conditions and equations in Theorem \ref{thm:as} and Theorem \ref{thm:as2} are satisfied.
\begin{proposition}\label{prop:1rel2}
 Let $w_i^{\circ}$ and $w_i^{\black}$ be the splitting weights for the 2--colour model. Let $w_{j,i+2-j}^{\circ}$ be the partitioning weights. Define a $1$--colour process with splitting weights $w_i=w_i^{\black}$ and partitioning weights $w_{j,i+2-j}=\frac{w_i^{\black}}{w_i^{\circ}}w_{j,i+2-j}^{\circ}$. Let $\rho_i^{\circ}$ and $\rho_i^{\black}$ be the degree densities of the 2--colour model, and let $a_i$ be the degree densities of the 1--colour model. If these are unique as solutions to the systems in Theorem \ref{thm:as} and Theorem \ref{thm:as2} respectively, then 
\begin{align}
\rho_i^{\circ}+\rho_i^{\black}=a_i
\end{align}
for all $1\leq i < \dmax+1$.
\end{proposition}

Let us illustrate Proposition \ref{prop:1rel2} by considering uniform partitioning weights in the the one--coloured and two--coloured cases, respectively. Indeed, this was the case considered in \cite{RNA} and mentioned at the beginning of the section. In any case, it can be shown that
\begin{align}
 \sum_{k=1}^{\infty}\left(\frac{2^k k}{e^2(k+2)!}+ \frac{2^k}{e^2(k+1)!} \right) = \frac{e^2-1}{2e^2},
\end{align}
so the asymptotic vertex densities are given by
\begin{align}
 \rho_k^{\circ} & = \frac{2e^2}{e^2-1}\frac{2^{k}k}{e^2(k+2)!} = \frac{2^{k+1}k}{(e^2-1)(k+2)!} \\
 \rho_k^{\black} & = \frac{2e^2}{e^2-1}\frac{2^k}{e^2(k+1)!} = \frac{2^{k+1}}{(e^2-1)(k+1)!}.
\end{align}
Now, if we follow the notation in Proposition \ref{prop:1rel2} we obtain a one--coloured process, which is precisely the uniform splitting model considered in Section \ref{sec:uniform}, with weights $w_k=k$ and partitioning weights 
\begin{align}
 w_{i,k+2-i}=\frac{w_k^{\black}}{w_k^{\circ}}w^{\circ}_{i,k+2-i}=\frac{k}{k+1}\cdot \frac{k+1}{\binom{k+1}{2}}=\frac{k}{\binom{k+1}{2}}=\frac{w_k}{\binom{k+1}{2}}.
\end{align}
Recall \eqref{eq:unif2}, i.e. that the limiting vertex densities in this case were
\begin{align}
 a_k = \frac{2^{k+2}(k+1)}{(e^2-1)(k+2)!}.
\end{align}
In this particular case one verifies that 
\begin{align}
 a_k= \frac{2^{k+2}(k+1)}{(e^2-1)(k+2)!}= \frac{2^{k+1}k}{(e^2-1)(k+2)!} + \frac{2^{k+1}}{(e^2-1)(k+1)!}=  \rho_k^{\circ}+  \rho_k^{\black},
\end{align}
as predicted by Proposition \ref{prop:1rel2}.

\end{document}